\documentclass[12pt,letterpaper]{amsart}

\usepackage{amsmath}
\usepackage{amssymb}
\usepackage{amsfonts}
\usepackage{amsthm}
\usepackage{subcaption}
\usepackage{tikz}
\usepackage{blkarray}
\usepackage{dsfont}

\newtheorem{theorem}{Theorem}
\newtheorem{prop}[theorem]{Proposition}
\newtheorem{lemma}[theorem]{Lemma}
\newtheorem{cor}[theorem]{Corollary}
\theoremstyle{definition}
\newtheorem{definition}[theorem]{Definition}
\newtheorem{example}[theorem]{Example}
\newtheorem{problem}[theorem]{Problem}
\newtheorem{remark}[theorem]{Remark}

\numberwithin{theorem}{section}

\newcommand{\Z}{\mathcal{Z}}
\newcommand{\R}{\mathbb{R}}
\newcommand{\bv}{\mathbf{v}}
\newcommand{\be}{\mathbf{e}}
\renewcommand{\O}{\mathcal{O}}

\renewcommand{\vert}{\text{vert}}
\newcommand{\inv}{\text{inv}}
\newcommand{\swap}{\text{swap}}
\newcommand{\Swap}{\text{Swap}}

\newcommand{\excl}{\text{excl}}
\newcommand{\bs}{\mathbf{s}}
\newcommand{\bw}{\mathbf{w}}
\renewcommand{\L}{\mathcal{L}}
\newcommand{\Zbb}{\mathbb{Z}}

\title[$h^*$-polynomial of the order polytope of the zig-zag poset]{The $h^*$-polynomial of the order polytope of the zig-zag poset}
\author{Jane Ivy Coons and Seth Sullivant}
\date\today

\begin{document}
\maketitle

\begin{abstract}
We describe a family of shellings for the canonical triangulation of the 
order polytope of the zig-zag poset. This gives a new combinatorial interpretation 
for the coefficients in the numerator of the Ehrhart series of this order polytope
in terms of the swap statistic on alternating permutations.
\end{abstract}

\section{Introduction and Preliminaries}

The \emph{zig-zag poset} $\Z_n$ on ground set $\{z_1, \dots, z_n\}$ is the poset with 
exactly the cover relations $z_1 < z_2 > z_3 < z_4 > \dots$. That is, this partial 
order satisfies $z_{2i-1} < z_{2i}$ and $z_{2i} > z_{2i+1}$ for all $i$ between $1$ 
and $\lfloor \frac{n-1}{2} \rfloor$. The order polytope of $\Z_n$, denoted $\O(\Z_n)$ 
is the set of all $n$-tuples $(x_1, \dots, x_n) \in \R^n$ that satisfy 
$0 \leq x_i \leq 1$ for all $i$ and $x_i \leq x_j$ whenever 
$z_i < z_j$ in $\Z_n$. In this paper, we introduce the ``swap" permutation statistic on alternating permutations
 to give a new combinatorial interpretation of the numerator of the Ehrhart series of $\O(\Z_n)$.

We began studying this problem in relation to combinatorial properties
of the Cavender-Farris-Neyman model with a molecular clock (or CFN-MC model) from mathematical phylogenetics \cite{coons2018}.
We were interested in the polytope associated to the toric variety
obtained by applying the discrete Fourier transform to the Cavender-Farris-Neyman model
with a molecular clock on a given rooted binary phylogenetic tree.
We call this polytope the CFN-MC polytope.
In particular, the Ehrhart polynomial of 
$\O(\Z_n)$ is equal to that of the CFN-MC polytope of any rooted binary tree on 
$n+1$ leaves .
Therefore, the Ehrhart series of $\O(\Z_n)$ is also equal to the Hilbert series 
of the toric ideal of phylogenetic invariants of the CFN-MC model on such a tree.

In the remainder of this section, we will give some preliminary definitions and key theorems
regarding alternating permutations, order polytopes and Ehrhart theory. 
In Section 2, we prove our main result, Theorem \ref{thm:ehrhartseries}, 
by giving a shelling of the canonical triangulation of the order polytope of the zig-zag poset.
In Section 3, we give an alternate proof of Theorem \ref{thm:ehrhartseries} by counting chains 
in the lattice of order ideals of the zig-zag poset. 
This proof makes use of the theory of Jordan-H\"older sets of general 
posets developed in Chapter 2 of \cite{ec1}.
In Section 4, we discuss some combinatorial properties of the swap statistic and present some open problems.

\subsection{Alternating Permutations}

\begin{definition}
An \emph{alternating permutation} on $n$ letters is a permutation 
$\sigma$ such that $\sigma(1) < \sigma(2) > \sigma(3) < \sigma(4) > \dots$. 
That is, an alternating permutation satisfies
$\sigma(2i-1) < \sigma(2i)$ and $\sigma(2i) > \sigma(2i+1)$ for $1 \leq i \leq \lfloor \frac{n}{2} \rfloor$.  
\end{definition}

We denote by $A_n$ the set of all alternating permutations.
Notice that alternating
permutations  coincide with order-preserving bijections from $[n]$ to $\Z_n$. 

The number of alternating permutations of length $n$ is the $n$th 
\emph{Euler zig-zag number} $E_n$. The sequence of Euler zig-zag numbers starting 
with $E_0$ begins $1, 1, 1, 2, 5, 16, 61, 272, \dots$. This sequence can be found in the 
Online Encyclopedia of Integer Sequences with identification number A000111 \cite{oeis}.
 The exponential generating function for the Euler zig-zag numbers satisfies
\[
\sum_{n \geq 0} E_n \frac{x^n}{n!} = \tan x + \sec x.
\]
Furthermore, the Euler zig-zag numbers satisfy the recurrence
\[
2 E_{n+1} = \sum_{k=0}^n \binom{n}{k} E_k E_{n-k}
\]
for $n \geq 1$ with initial values $E_0 = E_1 = 1$. A thorough background 
on the combinatorics of alternating permutations can be found in 
\cite{stanley2010}. The following new permutation statistic on 
alternating permutations is central to our results.

\begin{definition}
Let $\sigma$ be an alternating permutation. The permutation statistic 
$\swap(\sigma)$ is the number of $i < n$ such that $\sigma^{-1}(i) < \sigma^{-1}(i+1) - 1$. 
Equivalently, this is the number of $i < n$ such that $i$ is to the left of 
$i+1$ and swapping $i$ and $i+1$ in $\sigma$ yields another alternating permutation. 
The swap-set $\Swap(\sigma)$ is the set of all $i < n$ for which we can perform this 
operation. We say that $\sigma$ \emph{swaps to} $\tau$ if $\tau$ can be obtained 
from $\sigma$ by performing this operation a single time. 
\end{definition}

We will also make use of the following two features which can be defined for any permutation.
Let $\sigma \in S_n$.

\begin{definition}
A \emph{descent} of $\sigma$ is an index $i \in [n-1]$ such that $\sigma(i) > \sigma(i+1)$.
An \emph{inversion} of $\sigma$ is any pair $(i,j)$ for $1 \leq i < j \leq n$ such that $\sigma^{-1}(j) < \sigma^{-1}(i)$.
\end{definition}

When we write $\sigma$ in one-line notation, a descent is a position on $\sigma$ where the value of $\sigma$ drops.
An inversion is any pair of values in $[n]$ where the larger number appears before the smaller number in $\sigma$.
\subsection{Order Polytopes}

To every finite poset on $n$ elements one can associate a polytope in 
$\R^n$ by viewing the cover relations on the poset as inequalities on Euclidean space.

\begin{definition}
The \textit{order polytope} $\mathcal{O}(P)$ of any poset $P$ on ground set $p_1, \dots, p_n$ is the set of all $\bv \in \R^n$ that satisfy $0 \leq v_i \leq 1$ for all $i$ and $v_i \leq v_j$ if $p_i < p_j$ is a cover relation in $P$.
\end{definition}

Order polytopes for arbitrary posets have been the object of considerable study, 
and are discussed in detail in \cite{stanley1986}.  In the case of $\mathcal{O}(\Z_n)$, 
the facet defining inequalities are those of the form
\begin{align*} \begin{split}
-v_i  \leq 0 & \text{ for $i \leq n$ odd} \\
v_i  \leq 1 & \text{ for $i \leq n$ even}\\
v_i - v_{i+1}  \leq 0 & \text{ for $i \leq n-1$ odd, and} \\
-v_i + v_{i+1}  \leq 0 & \text{ for $ i \leq n-1$ even.} \end{split} \label{eqn:orderpolytopeineqs}
\end{align*}
Note that the inequalities of the form $-v_i \leq 0$ for $i$ even and 
$v_i \leq 1$ for $i$ odd are redundant. The order polytope of $\Z_n$ is 
also the convex hull of all $(v_1, \dots, v_n) \in \{0,1\}^n$ that correspond 
to labelings of $\Z_n$ that are weakly consistent with the partial order on $\{p_1, \dots, p_n\}$.

In \cite{stanley1986}, Stanley gives the following canonical 
unimodular triangulation of the order polytope of any poset $P$ on ground 
set $\{p_1, \dots, p_n\}$. Let $\sigma: P \rightarrow [n]$ be a linear extension of $P$. 
Denote by $\be_i$ the $i$th standard basis vector in $\R^n$. The simplex 
$\Delta^{\sigma}$ is the convex hull of $\bv_0^{\sigma}, \dots, \bv_n^{\sigma}$ where 
$\bv_0^{\sigma}$ is the all $1$'s vector and $\bv^\sigma_i = \bv^\sigma_{i-1} - \be_{\sigma^{-1}(i)}$. 
Letting $\sigma$ range over all linear extensions of $P$ yields a unimodular triangulation 
of $\O(P)$. Hence, the normalized volume of $\O(P)$ is the number of linear 
extensions of $P$. In particular, this means that the volume of $\O(\Z_n)$ 
is the Euler zig-zag number, $E_n$.

\begin{example}\label{ex:intro}
Consider the case when $n = 4$. The zig-zag poset $\Z_4$ is pictured in Figure 
\ref{fig:zigzagexample}.  The order polytope $\O(\Z_4)$ has facet defining inequalities

\begin{minipage}{.5\textwidth}
\begin{align*}
-v_1 &\leq 0 \\
-v_3 & \leq 0 \\
v_1 - v_2 & \leq 0 \\
v_3 - v_4 & \leq 0.
\end{align*}
\end{minipage}\begin{minipage}{.5\textwidth}
\begin{align*}
v_2 & \leq 1 \\
v_4 & \leq 1 \\
-v_2 + v_3 & \leq 0 \\
\end{align*}
\end{minipage}
The vertices of $\O(\Z_4)$ are the columns of the matrix
\[
\begin{bmatrix}
0 & 0 & 0 & 1 & 0 & 1 & 0 & 1 \\
0 & 1 & 0 & 1 & 1 & 1 & 1 & 1 \\
0 & 0 & 0 & 0 & 0 & 0 & 1 & 1\\
0 & 0 & 1 & 0 & 1 & 1 & 1 & 1 
\end{bmatrix}.
\]
The alternating permutations on 4 elements, which correspond to 
linear extensions of $\Z_4$ are $1324$, $1423$, $2314$, $2413$, and $3412$. 
Note that there are $E_4 = 5$ such alternating permutations, 
so the normalized volume of $\O(\Z_4)$ is 5. The simplex in the canonical 
triangulation of $\O(\Z_n)$ corresponding to $1423$ is
\[
\Delta^{1324} = \text{conv} \begin{bmatrix}
1 & 0 & 0 & 0 & 0 \\
1 & 1 & 1 & 1 & 0 \\
1 & 1 & 0 & 0 & 0 \\
1 & 1 & 1 & 0 & 0
\end{bmatrix}.
\]

\begin{figure}

\begin{center}
\begin{tikzpicture}
\draw(0,0) -- (1,1) -- (2,0) -- (3,1);
\draw[fill]  (0,0) circle [radius = .05];
\node[below] at (0,0) {$z_1$};
\draw[fill]  (1,1) circle [radius = .05];
\node[above] at (1,1) {$z_2$};
\draw[fill]  (2,0) circle [radius = .05];
\node[below] at (2,0) {$z_3$};
\draw[fill]  (3,1) circle [radius = .05];
\node[above] at (3,1) {$z_4$};
\end{tikzpicture}
\end{center}
\caption{The zig-zag poset $\Z_4$}
\label{fig:zigzagexample}
\end{figure}
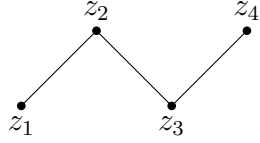
\end{example}

\subsection{Ehrhart Theory} \label{subsec:ehrhart}

We turn our attention to the study of Ehrhart 
functions and series of lattice polytopes. Let $P \subset \R^n$ be any 
polytope with integer vertices. Recall that the \emph{Ehrhart function}, 
$i_P(m)$, counts the integer points in dilates of $P$; that is,
\[
i_P(m) = \# (\Zbb^n \cap mP),
\]
where $mP = \{m \bv \mid \bv \in P\}$ denotes the $m$th dilate of $P$. 
The Ehrhart function is, in fact, a polynomial in $m$ \cite[Chapter~3]{beck2007}. 
We further define the \emph{Ehrhart series} of $P$ to be the generating function
\[
\mathrm{Ehr}_P(t) = \sum_{m \geq 0} i_P(m) t^m.
\] 
The Ehrhart series is of the form
\[
\mathrm{Ehr}_P(t) = \frac{h^*_P(t)}{(1-t)^{d+1}},
\]
where $d$ is the dimension of $P$ and 
$h^*_P(t)$ is a polynomial in $t$ of degree at most $n$. Often
we just write $h^*(t)$ when the particular polytope is clear.  The coefficients of 
$h^*(t)$ have an interpretation in terms of a \emph{shelling} of a 
unimodular triangulation of $P$, if such a shellable unimodular triangulation exists.

\begin{definition}
Let $\mathcal{T}$ be the collection of maximal dimensional 
simplices in a pure simplicial complex of dimension $d$ with $\#\mathcal{T} = s$. 
An ordering $\Delta_1, \Delta_2, \dots, \Delta_s$ on the simplices in 
$\mathcal{T}$ is a \emph{shelling order} if for all $1 < r \leq s$,
 \[
 \bigcup_{i=1}^{r-1} \Big( \Delta_i \cap \Delta_r \Big)
 \]
 is a union of facets of $\Delta_r$.
 \end{definition}

Equivalently, the order $\Delta_1, \Delta_2, \dots, \Delta_s$ 
is a shelling order if and only if for all $r \leq s$ and $k < r$, there exists an $
i < r$ such that $\Delta_k \cap \Delta_r \subset \Delta_i \cap \Delta_r$ and 
$\Delta_i \cap \Delta_r$ is a facet of $\Delta_r$. This means that when we build our 
simplicial complex by adding facets in the order prescribed by the shelling order, 
we add each simplex along its highest dimensional facets. 
Keeping track of the number of facets that each simplex is added along gives 
the following relationship between shellings of a triangulation of an 
integer polytope and the Ehrhart series of the polytope, which is proved in \cite[Chapter~3]{beck2007}. 
 
 \begin{theorem}\label{thm:hstarshelling}
Let $P$ be a polytope with integer vertices. Let $\{\Delta_1, \dots, \Delta_s\}$ be a unimodular triangulation of $P$ using no new vertices. Denote by $h^*_j$ the coefficient of $t^j$ in the $h^*$ polynomial of $P$. If $\Delta_1, \dots, \Delta_s$ is a shelling order, then $h^*_j$ is the number of $\Delta_i$ that are added along $j$ of their facets in this shelling. Equivalently,
\[
 h^*(t) = \sum_{i=1}^s t^{a_i},
 \]
 where $a_i = \#\{k < i \mid \Delta_k \cap \Delta_i \text{ is a facet of } \Delta_i \}$.
 \end{theorem}

\begin{example}
Consider the order polytope $\O(\Z_4)$ with its canonical triangulation  
by alternating permutations
\[
\Delta^{3412}, \Delta^{2413},  \Delta^{2314}, \Delta^{1423}, \Delta^{1324}.
\]
This particular ordering of the facets in the canonical triangulation is a
special case of the shelling order that will be established and proved in the
next section.
The fact that this is a shelling order can be checked directly in this
example, for instance:
\[
\Delta^{2314}  \cap (\Delta^{3412} \cup \Delta^{2413})   =  \mathbf{conv}
\begin{bmatrix}
1  & 1 &  0 & 0  \\
1  & 1 &  1 & 0  \\
1  & 0 &  0 & 0  \\
1  & 1 &  1 & 0  
\end{bmatrix}
\]
which is a facet of $\Delta^{2314}$.  Since the intersection consists of a single facet, it will
contribute a $1$ to the coefficient of $t$ in $h^*_{\O(\Z_4))}(t)  =  1 + 3t + t^2$.
\end{example}

The goal of this paper is to prove the following theorem relating 
the $h^*$-polynomial of $\O(\Z_n)$ and the swap statistic.

\begin{theorem}\label{thm:ehrhartseries}
The numerator of the Ehrhart series of $\O(\Z_n)$ is
\[
h^*_{\O(\Z_n)}(t) = \sum_{\sigma \in A_n} t^{\text{swap}(\sigma)}.
\]
\end{theorem}

\begin{remark}
Alternate formulas for the $h^*$-polynomial of the order polytope of a poset $P$ exist,
as described in \cite[Chapter~3.13]{ec1}.
However, many of these formulas refer to the \emph{Jordan-H\"older set} of $P$
and in particular, descents in the permutations in this set,
which we will discuss in more detail in Section \ref{sec:rankselection}.
In the case of $\Z_n$, the elements of this Jordan-H\"older set are not alternating or inverse alternating permutations
because they arise from linear extensions with respect to a natural labeling of $\Z_n$.
The elements of the Jordan-H\"older set do not have as nice of a combinatorial description as the alternating permutations,
and there is not an obvious bijection between swaps in alternating permutations and descents in elements of
the Jordan-H\"older set.
\end{remark}


\section{Shelling the Canonical Triangulation of the Order Polytope}\label{sec:shelling}

In this section we describe a family of shelling orders on the simplices
of the canonical triangulation of $\O(\Z_n)$.  
Let $\sigma$ be an alternating permutation. We will denote by 
$\vert(\sigma)$ the set of all vertices of the simplex $\Delta^\sigma$.
 Note that this is the set of all $0/1$ vectors $\bv$ of length $n$ 
 that have $v_i \leq v_j$ whenever $\sigma(i) < \sigma(j)$.

\begin{prop}\label{prop:swapfacets}
The simplices $\Delta^{\sigma}$ and $\Delta^{\tau}$ are joined along a facet 
if and only if $\sigma$ swaps to $\tau$ or $\tau$ swaps to $\sigma$.
\end{prop}

\begin{proof}
Simplices $\Delta^{\sigma}$ and $\Delta^{\tau}$ are joined along a facet 
if and only if $\vert(\sigma)$ and $\vert(\tau)$ differ by a single element. 
Since every simplex in the canonical triangulation of $\O(\Z_n)$ has exactly 
one vertex with the sum of its components equal to $i$ for $0 \leq i \leq n$ 
and the all $0$'s and all $1$'s vector are in every simplex in this triangulation, 
this occurs if and only if there exists an $i$ with $1 \leq i \leq n-1$ such that 
$\vert(\sigma) - \{ \bv_i^{\sigma} \} = \vert(\tau) - \{\bv_i^{\tau} \}$. 
By definition of each $\bv_j^{\sigma}$ and $\bv_j^{\tau}$, this occurs if and only 
if $\sigma^{-1}(j) = \tau^{-1}(j)$ for all $j \neq i, i+1$ and 
$\be_{\sigma^{-1}(i)} + \be_{\sigma^{-1}(i+1)} = \be_{\tau^{-1}(i)} + \be_{\tau^{-1}(i+1)}$. 
This is true if and only if swapping the positions of $i$ and $i+1$ in $\sigma$ yields $\tau$, as needed.
\end{proof}

Denote by $\inv(\sigma)$ the number of inversions of a permutation $\sigma$; 
that is, $\inv(\sigma)$  is the number of pairs $i < j$ such that $\sigma(i) > \sigma(j)$. 
We similarly define a \emph{non-inversion} to be a pair $i < j$ with $\sigma(i) < \sigma(j)$. 
We call an inversion or non-inversion $(i,j)$ \emph{relevant} if $i < j-1$; 
that is, if it is not required by the structure of an alternating permutation. 
Note that performing a swap on an alternating permutation always increases 
its inversion number by exactly one.

The following lemma relates relevant non-inversions to swaps in between them.

\begin{lemma}\label{lem:swapbetweennoninv}
Let $\sigma$ be an alternating permutation.
Let $a, b \in [n]$ such that $(\sigma^{-1}(a), \sigma^{-1}(b))$ is a relevant non-inversion of $\sigma$.
Then there exists a $k$ with $a \leq k < b$ such that $k$ is a swap of $\sigma$.
\end{lemma}

\begin{proof}
We proceed by induction on $b-a$. 
If $b-a = 1$, then since $(i,j)$ is a relevant 
non-inversion, $a$ is a swap in $\sigma$.

Let $b-a > 1$. Consider the position of $a+1$ in $\sigma$. There are three cases.
If $\sigma^{-1}(a+1) < \sigma^{-1}(b)-1$, then $(\sigma^{-1}(a+1), \sigma^{-1}(b))$ 
is a relevant non-inversion, and we are done by induction. If $\sigma^{-1}(a+1) >\sigma^{-1}(b)$, 
then $a$ is a swap in $\sigma$. If $\sigma^{-1}(a+1) = \sigma^{-1}(b)-1$, 
then note that $\sigma^{-1}(a) < \sigma^{-1}(a+1)-1$ since otherwise, 
$a, a+1, b$ would be an adjacent increasing 
sequence in $\sigma$, which would contradict that $\sigma$ is alternating. 
So $a$ is a swap in $\sigma$, as needed.
\end{proof}

Theorem \ref{thm:ehrhartseries} follows as a 
corollary of Theorem \ref{thm:hstarshelling}, Proposition \ref{prop:swapfacets} and the following theorem.

\begin{theorem}\label{thm:shellingorder}
Let $\sigma_1, \dots, \sigma_{E_n}$ be an order on the alternating permutations 
such that if $i < j$ then $\inv(\sigma_i) \geq \inv(\sigma_j)$. 
Then the order $\Delta^{\sigma_1}, \dots, \Delta^{\sigma_{E_n}}$ 
on the simplices of the canonical triangulation of $\O(\Z_n)$ is a shelling order.
\end{theorem}

Note that since performing a swap increases
inversion number by exactly one,
the condition of Theorem \ref{thm:shellingorder} implies that
if $\sigma_j$ swaps to $\sigma_i$, then $i < j$.
For any alternating permutation $\sigma$, define the 
\emph{exclusion set} of $\sigma$, $\excl(\sigma)$ to be the set of all 
$\bv_k^{\sigma} \in \vert(\sigma)$ such that $k$ is a swap in $\sigma$. In other words,
\[
\excl(\sigma)  = \{\bv \mid \bv \in \Delta^{\sigma} - \Delta^{\tau} \text{ for some } \tau \text{ such that }  \sigma \mbox{ swaps to } \tau  \}. 
\]
 In the proof of Theorem \ref{thm:shellingorder}, we will show that Proposition \ref{prop:swapfacets} implies that in order to prove Theorem \ref{thm:shellingorder}, it suffices to check that if $\inv(\sigma) \leq \inv(\tau)$, then $\excl(\sigma) \not\subset \vert(\tau)$. This fact follows from the next two propositions.

\begin{prop}\label{prop:maxinversion}
An alternating permutation $\sigma$ maximizes inversion number over all alternating permutations $\tau$ with $\excl(\sigma) \subset \vert(\tau)$.
\end{prop}

\begin{proof}
Consider a vertex $\bv_k^{\sigma} \in \vert(\sigma)$. Note that we may read all of the
non-inversions $(i,j)$ with $\sigma(i) \leq k < \sigma(j)$ from $\bv_k^{\sigma}$ 
since these correspond to pairs of positions in $\bv_k^{\sigma}$ with a $0$ in the 
first position and a $1$ in the second. That is to say, we have $\bv_k^{\sigma}(i) = 0$, 
$\bv_k^{\sigma}(j) = 1$, and $i < j$.

We claim that every relevant non-inversion of $\sigma$ can be read from 
an element of $\excl(\sigma)$ in this way.
By Lemma \ref{lem:swapbetweennoninv}, there exists a swap $k$ in $\sigma$ with 
$\sigma(i) \leq k < \sigma(j)$, and the relevant non-inversion 
$(i,j)$ can be read from $\bv_k^{\sigma}$ in the manner described above. 

Therefore, all relevant non-inversions in $\sigma$ can be found as a 
non-adjacent $0-1$ pair in a vertex in $\excl(\sigma)$. In particular, 
we can count the number of relevant non-inversions in $\sigma$ from the 
vertices in $\excl(\sigma)$. Furthermore, if $\excl(\sigma) \subset \vert(\tau)$, 
then all non-inversions in $\sigma$ must also be non-inversions in $\tau$, 
though $\tau$ can contain more non-inversions as well. So $\sigma$ minimizes 
the number of non-inversions, and therefore maximizes the number of inversions, 
over all $\tau$ with $\excl(\sigma) \subset \vert(\tau)$.
\end{proof}

\begin{prop}\label{prop:uniquemax}
Let $S \subset \vert(\O(\Z_n))$ be contained in 
$\vert(\sigma)$ for some alternating $\sigma$.  Then there exists a unique 
alternating $\hat{\sigma}$ that maximizes inversion number over all 
alternating permutations whose vertex set contains $S$. 
\end{prop}

\begin{proof}
Let $S = \{ \bs_0, \bs_1, \dots, \bs_r \}$ ordered by decreasing coordinate sum. 
We can assume that $S$ contains both the all zeroes and all ones vectors since those
vectors belong to the simplex $\Delta^\sigma$ for any alternating permutation $\sigma$.
Since $S \subset \vert(\sigma)$ for some alternating $\sigma$, 
if $\bs_i(j) = 0$, then $\bs_k(j) = 0 $ for all $k > i$. 
For $i = 1, \ldots, r$, let $m_i$ be the number of positions in $\bs_i$ that 
are equal to zero, and let $n_i = m_i - m_{i-1}$ (with $n_1 = m_1$).

Let $\tau$ be any alternating permutation such that $S \subseteq \vert(\tau)$.  
The $0$-pattern of 
each $\bs_i$ partitions the entries of all $\tau$ with $S \subset \vert(\tau)$ as follows:
For $1 \leq k \leq r$, the $n_k$ positions $j$ such that $\bs_k(j) = 0$ and $\bs_{k-1}(j) = 1$ 
are the positions of $\tau$ such that $\tau(j) \in \{m_{k-1} + 1, \dots, m_k \}$.

The positions of inversions and non-inversions across these groups are 
fixed for all $\tau$ with $S \subset \vert(\tau)$. We can build an alternating 
permutation $\hat{\sigma}$ that maximizes the inversions within each group as follows. 
For $1 \leq k \leq r$, let $j^k_1, \dots, j^k_{n_k}$ be the positions of $\hat{\sigma}$ 
that must take values in $\{ m_{k-1} +1, \dots, m_k \}$, as described above. 
We place these values in reverse; i.e.~map $j_l^k$ to $m_k - l + 1$. 
The permutation obtained in this way need not be alternating, 
so we switch adjacent positions that need to contain non-descents 
in order to make the permutation alternating. Note that we never need to 
make such a switch between groups, since the partition given by $S$ respects 
the structure of an alternating permutation. 

This permutation is unique because within the $k$th group, arranging the 
values in this way is equivalent to finding the permutation on $n_k$ elements 
with some fixed non-descent positions that maximizes inversion number. 
To obtain this permutation, we begin with the permutation $m_k \ m_k-1 \dots m_{k-1}+1$ 
and switch all the positions that must be non-descents. The alternating structure of 
the original permutation implies that none of these non-descent positions
can be adjacent, so these transpositions commute and give a unique permutation.
\end{proof}

\begin{example}\label{ex:uniquemax}
Let $n = 7$ and let
\[
S = \left\{ \begin{bmatrix}
1 \\ 1 \\ 1\\ 1\\ 1\\ 1\\ 1\\
\end{bmatrix}, \begin{bmatrix}
0 \\ 1 \\ 0 \\ 1 \\ 1 \\ 1 \\ 0
\end{bmatrix}, \begin{bmatrix}
0 \\ 1 \\ 0 \\ 0 \\ 0 \\ 0 \\ 0
\end{bmatrix}, \begin{bmatrix}
0\\ 0\\ 0\\ 0\\ 0\\ 0\\ 0\\
\end{bmatrix} \right\}
\]
We will construct $\hat{\sigma}$, the alternating permutation that maximizes 
inversion number overall alternating permuations whose vertex set contains $S$.
The second and third vertices in $S$ are the only one that gives information about the position
of each character; we will denote them $\bw_1$ and $\bw_2$, respectively. Since $\bw_1$ 
has $0$'s in exactly the first, third and seventh positions,
we know that $1,2$ and $3$ are in these positions. We insert them into these positions
in decreasing order, so that $\hat{\sigma}$ has the form
\[
3 \ \underline{\hspace{.25cm}} \ 2 \ \underline{\hspace{.25cm}} \  \ \underline{\hspace{.25cm}} \  \ \underline{\hspace{.25cm}} \  1.
\]

The zeros added in $\bw_2$ are in the fourth, fifth and sixth positions. 
Placing them in decreasing order yields the permutation
\[
3 \ \underline{\hspace{.25cm}} \ 2 \ 6 \ 5 \ 4 \ 1.
\]
However, this permutation cannot be alternating, since there must be an ascent from position 5 to position 6.
To create this ascent, we switch the entries in these positions, yielding a permutation of the form
\[
3 \ \underline{\hspace{.25cm}} \ 2 \ 6 \ 4 \ 5 \ 1.
\]

Finally, the only character missing is 7, which must go in the remaining space. This gives the permutation
\[
\hat{\sigma} = 3 \ 7 \ 2 \ 6 \ 4 \ 5 \ 1.
\]
\end{example}

\begin{proof}[Proof of Theorem \ref{thm:shellingorder}]
First, we claim that it suffices to show that for any alternating permutations $\sigma$ and $\tau$, if $\inv(\tau) \geq \inv(\sigma)$ then $\excl(\sigma) \not\subset \vert(\tau)$. 
Indeed, for any $\rho$ with $\inv(\rho) \geq \inv(\sigma)$,  by Proposition \ref{prop:swapfacets} we have that $\Delta_{\sigma} \cap \Delta_{\rho}$
is a facet of $\Delta_{\sigma}$ if and only if $\sigma$ swaps to $\rho$.
This is the case if and only if $\Delta_{\sigma} \cap \Delta_{\rho} = \Delta_{\sigma} \setminus \{ \bv_i \}$
for some $\bv_i \in \excl(\sigma)$.
So if $\Delta_{\sigma} \cap \Delta_{\tau} \not\subset \Delta_{\sigma} \cap \Delta_{\rho}$ for any $\rho$ 
such that $\inv(\rho) \geq \inv(\sigma)$ with $\Delta_{\sigma} \cap \Delta_{\rho}$ a facet of $\Delta_{\sigma}$,
then we must have $\excl(\sigma) \subset \vert(\tau)$.
The contrapositive of this statement shows that if $\excl(\sigma) \not\subset \vert(\tau)$, then the given order on
the facets of the triangulation is a shelling.

If $\inv(\tau) > \inv(\sigma)$, then since $\sigma$ maximizes inversion number over all alternating permutations that contain the exclusion set of $\sigma$ by Proposition \ref{prop:maxinversion}, $\excl(\sigma) \not\subset \vert(\tau)$. Furthermore, Proposition \ref{prop:uniquemax} implies that if $\inv(\tau) = \inv(\sigma)$, then $\excl(\sigma) \not\subset \vert(\tau)$ because $\sigma$ is the \emph{unique} permutation that maximizes inversion number of all alternating permutation that contain its exclusion set.
\end{proof}

\begin{proof}[Proof of Theorem \ref{thm:ehrhartseries}]
Let $\Delta^{\sigma_1}, \dots, \Delta^{\sigma_{E_n}}$ be a shelling order as described in Theorem \ref{thm:shellingorder}. Then by Proposition \ref{prop:swapfacets}, each $\Delta^{\sigma_i}$ is added in the shelling along exactly $\swap(\sigma_i)$ facets. Therefore, by Theorem \ref{thm:hstarshelling}, 
\[
h^*_{\O(\Z_N)}(t) = \sum_{\sigma \in A_n} t^{\swap(\sigma)},
\]
as needed.
\end{proof}

We conclude this section by remarking that not all of the shellings described
in Theorem \ref{thm:shellingorder} can be obtained from EL- or CL-labelings
of the lattice of order ideals of $\Z_n$.
Denote by $J(\Z_n)$ the distributive lattice of order ideals of $\Z_n$ ordered by inclusion.
Saturated chains in $J(\Z_n)$ are in bijection with elements of $A_n$ via the map that sends
an alternating permutation $\sigma$ to the chain of order ideals,
\[
I_0 \subsetneq I_1 \subsetneq I_2 \subsetneq \dots \subsetneq I_n
\]
where $I_j = \{ \sigma^{-1}(1), \dots, \sigma^{-1}(j) \}$ \cite[Chapter~3.5]{ec1}. 

\begin{definition}
Let $P$ be a graded bounded poset and let $E(P)$ be the set of cover relations of $P$. 
An \emph{EL-labeling} of $P$ is a labeling $\lambda$ of $E(P)$ with integers such that
\begin{itemize}
\item each closed interval $[a,b]$ of $P$ has a unique $\lambda$-increasing saturated chain, and
\item this $\lambda$-increasing chain lexicographically precedes all other saturated chains from $a$ to $b$.
\end{itemize}
A poset that has an EL-labeling is called \emph{EL-shellable}.
\end{definition}

For more details on poset shellability, we refer the reader to \cite{wachs2007}. 
If $P$ is EL-shellable with EL-labeling $\lambda$, then lexicographic order
on the saturated chains of $P$ with respect
to $\lambda$ gives a shelling of the order complex of $P$ \cite{bjorner1980}.
In the case of $J(\Z_n)$, its order complex is isomorphic to
the canonical triangulation of the order polytope $\O(\Z_n)$
via the bijection described above.
So finding EL-labelings of $J(\Z_n)$ is one way to construct
shellings of the canonical triangulations of $\O(\Z_n)$.
However, not all of the shellings described 
in Theorem \ref{thm:shellingorder} can be obtained in this way.

\begin{prop}\label{prop:ELshelling}
There exist shelling orders on the canonical triangulation of $\O(\Z_n)$
given by the conditions of Theorem \ref{thm:shellingorder} that cannot be obtained
from EL-labelings of $J(\Z_n)$.
\end{prop}

\begin{proof}
For the sake of readability, we discuss these shellings on the level of alternating permutations.
The ``position of $\sigma$ in a shelling order" is taken to mean the position of $\Delta^{\sigma}$ in that
shelling order for the canonical triangulation of $\O(\Z_n)$.
All of the shellings described in Theorem \ref{thm:shellingorder} begin with the unique
alternating permutation $\bar\sigma$ that maximizes inversion number over all alternating permutations;
this permutation exists by Proposition \ref{prop:uniquemax}.
Note that $\bar\sigma^{-1}(1)$ is $n-1$ or $n$, depending upon the parity of $n$.

We address the case where $n \geq 5$ is odd.
Then $\bar\sigma$ is of the form 
\[ 
\bar\sigma = (n-1) \ n \ (n-3) \ (n-2) \ \dots \ 4 \ 5 \ 2 \ 3 \ 1.
\]
Let $\lambda$ be an EL-labeling of the cover relations of $J(\Z_n)$ that induces
a shelling with $\bar\sigma$ is its first element. (Note that if no such EL-labeling exists,
the proposition holds trivially.)

The sets $\emptyset, \{p_n\}, \{p_{n-2}\}$ and $\{p_{n-2}, p_n \}$ are order ideals of $J(\Z_n)$
that comprise the interval $[\emptyset, \{p_{n-2}, p_n \}]$.  The chain corresponding to $\bar\sigma$ begins with
$\emptyset \lessdot \{ p_n \} \lessdot \{p_{n-2}, p_n \}$.
So this must must be the unique $\lambda$-increasing chain in the interval $[\emptyset, \{p_{n-2}, p_n \}]$.
As such, it lexicographically precedes the chain $\emptyset \lessdot \{ p_{n-2} \} \lessdot \{p_{n-2}, p_n \}$.
This implies that any permutation $\sigma$ with $\sigma^{-1}(1) = n$ and $\sigma^{-1}(2) = n-2$ will
precede any permutation $\tau$ with $\tau^{-1}(2) = n$ and $\tau^{-1}(1) = n-2$.

In particular, let $\sigma$ be obtained from $\bar\sigma$ by switching the positions of $3$ and $4$.
Let $\tau$ be obtained from $\bar\sigma$ by switching the positions of $1$ and $2$. 
Then in any EL-shelling, $\sigma$ will come before $\tau$.
However, $\sigma$ and $\tau$ have the same inversion number and have exactly one swap position,
so they are interchangeable in any order given by the conditions of Theorem \ref{thm:shellingorder}.

An analogous argument works when $n \geq 6$ is even, and can be adapted for the case when $n=4$.
\end{proof}

This proposition and proof can also be adapted to show 
that not all shellings arising from Theorem \ref{thm:shellingorder}
can be obtained from CL-labelings of $J(\Z_n)$.


\section{The Swap Statistic Via Rank Selection}\label{sec:rankselection}

An alternate proof of Theorem \ref{thm:ehrhartseries} relies 
heavily on the concepts of rank selection and flag $f$-vectors 
developed for general posets in Sections 3.13 and 3.15 of 
\cite{ec1}. We will focus our attention to the zig-zag poset, $\Z_n$. 
Denote by $J(\Z_n)$ the distributive lattice of order ideals in $\Z_n$ ordered by inclusion. 
Let $S = \{s_1, \dots, s_k\} \subset [0,n]$, where $[0,n] = \{0,\dots,n\}$. 
We always assume that $s_1 < s_2 < \ldots < s_k$.
Denote by $\alpha_n(S)$ the number of chains of order ideals 
$I_{1} \subsetneq \dots \subsetneq I_{k}$ in $J(\Z_n)$ such that $\#I_{j} = s_j$ for all $j$. Define
\[
\beta_n(S) = \sum_{T \subset S} (-1)^{\#(S - T)}\alpha_n(T).
\]
By the Principle of Inclusion-Exclusion, or equivalently, 
via M\"obius inversion on the Boolean lattice,
\[
\alpha_n(S) = \sum_{T \subset S} \beta_n(S).
\]

In Section 3.13 of \cite{ec1}, the function $\alpha_n : 2^{[0,n]} \rightarrow \Zbb$ 
is called the \emph{flag f-vector} of $\Z_n$ and $ \beta_n: 2^{[0,n]} \rightarrow \Zbb$ 
is called the \emph{flag h-vector} of $\Z_n$. For any poset $P$ of size $n$, 
let $\omega: P \rightarrow [n]$ be an order-preserving bijection that assigns 
a label to each element of $P$; in this case, $\omega$ is called a \emph{natural labeling}.
Then for any linear extension 
$\sigma: P \rightarrow [n]$, we may define a permutation of the labels 
by $\omega(\sigma^{-1}(1)),\dots,\omega(\sigma^{-1}(n))$. 
The \emph{Jordan-H\"older set} $\L(P,\omega)$ is the set of all 
permutations obtained in this way. The following result for arbitrary 
finite posets can be found in chapter 3.13 of \cite{ec1}.

\begin{theorem}[\cite{ec1}, Theorem 3.13.1]
Let $S \subset [n-1]$. Then $\beta_n(S)$ is equal to the number 
of permutations $\tau \in \L(P,\omega)$ with descent set $S$.
\end{theorem}

The \emph{order polynomial} of a poset $P$, $\Omega_P(m)$ is the number of order preserving maps from $P$ to $[m]$.
The Ehrhart polynomial of $\O(\Z_n)$ evaluated at
$m$ is equal to the order polynomial of $\Z_n$ evaluated at $m+1$ \cite{stanley1986}. 
We also have the following equality of generating functions
from Theorem 3.15.8 of \cite{ec1}.
We restate the relevant special case of this theorem here.

\begin{theorem}[\cite{ec1}, Theorem 3.15.8]
Let $\omega: P \rightarrow [n]$ be an order-preserving bijection.
Then
\[
\sum_{m \geq 0} \Omega_P(m) x^m = \frac{\sum_{\sigma \in \L(P,\omega)} x^{1+\mathrm{des}(\sigma)}}{(1-x)^{p+1}},
\]
where $p$ is the cardinality of $P$.
\end{theorem}

Therefore, since $i_{\O(\Z_n)}(m) = \Omega_{\O(\Z_n)}(m+1)$, we have that
\[
\mathrm{Ehr}_{\O(\Z_n)}(t) = \frac{\sum_{\sigma \in \L(\O(\Z_n),\omega)} x^{\mathrm{des}(\sigma)}}{(1-x)^{n+1}}.
\]
It follows that the $h^*$-polynomial of $\O(\Z_n)$ is
\begin{equation}\label{eqn:hstarresult}
h^*_{\O(\Z_n)}(t) = \sum_{S \subset [n-1]} \beta_n(S) t^{\#S}.
\end{equation}

So, Theorem \ref{thm:ehrhartseries} will follow from 
Equation \ref{eqn:hstarresult} and the following theorem, 
which is analogous to Theorem 3.13.1 in \cite{ec1}.

\begin{theorem}\label{thm:betanumbers}
Let $S \subset [n-1]$. Then $\beta_n(S)$ is the number of 
alternating permutations $\omega$ with $\Swap(\omega) = S$.
\end{theorem}

To prove this theorem, for every $S = \{s_1, \dots, s_n\} \subset [n-1]$, 
we will find define a function $\phi_S$ that maps chains of order ideals of 
sizes $s_1, \dots, s_k$ to alternating permutations whose swap set is 
contained in $S$. Let $I_1, \dots, I_k$ be a chain of order ideals in 
$J(\Z_n)$ with sizes $\#I_j = s_j$. Let $\bw_i$ be the vertex of $\O(\Z_n)$
that satisfies
\[
\bw_i(j) = \begin{cases}
0 & \text{ if } j \in I_i \\
1 & \text{ if } j \not\in I_i.
\end{cases}
\]
Define $\phi_S(I_1, \dots, I_k)$ to be the unique alternating permutation
that maximizes inversion number over all alternating permutations whose vertex
set contains $\{\bw_1, \dots, \bw_k\}$. This map is well-defined by Proposition \ref{prop:uniquemax}.

Let $\psi_S$ be the map that sends an alternating permutation 
$\omega$ with $\Swap(\omega) \subset S$ to the chain of order ideals 
$(I_1, \dots,I_k)$ where each $I_j = \{\omega^{-1}(1), \dots, \omega^{-1}(s_j)\}$.  
Since every alternating permutation $\omega$ is a linear extension of $\Z_n$, 
each $I_j$ obtained in this way is an order ideal. They form a chain by construction, 
so the map $\psi_S$ is well-defined.
We will show that $\psi_S$  is the inverse of $\phi_S$ in the proof of Theorem \ref{thm:betanumbers}.

\begin{example}
Consider the zig-zag poset on seven elements $\Z_7$ pictured in Figure \ref{fig:Z7}. 
Let $S = \{3, 6\}$, and let $I_1 = \{a, c, g\}$ and $I_2 = \{a, c, d, e, f ,g\}$ be the given 
order ideals of sizes 3 and 6 respectively. Then the vectors $\bw_1$ and $\bw_2$ are
\[
\bw_1 = \begin{bmatrix}
0 \\ 1 \\ 0 \\ 1 \\ 1 \\ 1 \\ 0
\end{bmatrix} \qquad \text{ and } \qquad
\bw_2 = \begin{bmatrix}
0 \\ 1 \\ 0 \\ 0 \\ 0 \\ 0 \\ 0
\end{bmatrix}.
\]

Notice that these are the same vectors $\bw_1$ and $\bw_2$ as in Example \ref{ex:uniquemax}.
So the unique alternating permutation $\phi_S(I_1, I_2)$ that maximizes
inversion number over all alternating permutations whose vertex set contains
$\{ \bw_1, \bw_2 \}$ is the same permutation as in Example \ref{ex:uniquemax},
\[
\phi_S(I_1, I_2) = 3 \ 7 \ 2 \ 6 \ 4 \ 5 \ 1.
\]
Note that $\Swap(3726451)  = \{3\}  \subset \{3,6\} = S.$  

Now let $\omega = 3726451$. We will recover our original order ideals $I_1$ and $I_2$ by finding $\psi_S(\omega)$. For clarity, we will treat $\omega$ as a map from $\{a,\dots,g\}$ to $\{1, \dots,7\}$. The first order ideal of $\psi_S(\omega)$ consists of the inverse images of $1, 2$, and $3$ in $\omega$. That is,
\[
I_1 = \{ \omega^{-1}(1), \omega^{-1}(2), \omega^{-1}(3)\} = \{a, c, g\}.
\]
The second order ideal of $\psi_S(\omega)$ consists of the inverse images of $1$ through $6$ in $\omega$. So we obtain
\[
I_2 = \{\omega^{-1}(1), \dots, \omega^{-1}(6)\} = \{a, c, d, e, f, g\}.
\]
Note that this is, in fact, the chain of order ideals with which we began.
\end{example}

\begin{figure}
\centering
\begin{tikzpicture}[scale=.5]
\draw[fill] (0,0) circle [radius = .1];
\node[below] at (0,0) {$a$};
\draw[fill] (1,1) circle [radius = .1];
\node[above] at (1,1) {$b$};
\draw[fill] (2,0) circle [radius = .1];
\node[below] at (2,0) {$c$};
\draw[fill] (3,1) circle [radius = .1];
\node[above] at (3,1) {$d$};
\draw[fill] (4,0) circle [radius = .1];
\node[below] at (4,0) {$e$};
\draw[fill] (5,1) circle [radius = .1];
\node[above] at (5,1) {$f$};
\draw[fill] (6,0) circle [radius = .1];
\node[below] at (6,0) {$g$};
\draw (0,0) -- (1,1) -- (2,0) -- (3,1) -- (4,0) -- (5,1) -- (6,0);
\end{tikzpicture}
\caption{The zig-zag poset $\Z_7$} \label{fig:Z7}
\end{figure}
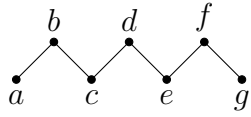

\begin{proof}[Proof of Theorem \ref{thm:betanumbers}]
Let $S = \{s_1, \dots, s_k\} \subset [n-1]$. 
We will show that $\alpha_n(S)$ is the number of 
alternating permutations whose swap set is contained in 
$S$ by showing that the map $\phi_S$ described above is a bijection.

Let $I_1, \dots, I_k$ be a chain of order ideals in $J(\Z_n)$ with sizes $\#I_j = s_j$. 
It is clear from the definitions of $\phi_S$ and $\psi_S$ that 
\[
\psi_S(\phi_S(I_1,\dots,I_k)) = (I_1,\dots,I_k).\]
 Since $\phi_S$ is injective, it suffices to show that 
 $\psi_S$ is also injective. We will show that 
 $\phi_S(I_1,\dots,I_k)$ is the only alternating permutation that maps to 
 $(I_1, \dots, I_k)$ under $\psi_S$. 

Since $\omega = \phi_S(I_1,\dots,I_k)$ is the unique alternating permuation
that maximizes inversion number over all alternating permutations with
$\{\bw_1, \dots, \bw_k\}$ in their vertex sets, any other alternating permutation
$\sigma$ that maps to $(I_1, \dots, I_k)$ under $\psi_S$ must have fewer
inversions than $\omega$.

Let $\sigma$ be such a permutation. Since each inversion between the sets $I_1$,
$\Z_n - I_k$ and $I_j - I_{j-1}$ for all $1 < j \leq k$ are fixed, the additional
non-inversion must be contained in one of these sets. Without loss of generality,
let this be $R = I_j - I_{j-1}$. Denote by $\sigma|_R$ the restriction of $\sigma$
to the domain $R$.
Let $(\sigma^{-1}(a), \sigma^{-1}(b))$ be the non-inversion of 
$\sigma|_R$ that is \emph{not} required by the alternating structure. 
Then by Lemma \ref{lem:swapbetweennoninv}, there exists a $k$ such that
$a \leq k < b$ and $k$ is a swap in $\sigma$. Since $a \leq k < b$,
$\sigma^{-1}(k)$ and $\sigma^{-1}(k+1)$ are in $R$, 
so $k$ is also a swap in $\sigma|_R$ as well. 
So the swap set of $\sigma$ is not contained in $S$ and we have reached a contradiction.

Therefore, $\omega$ is the only alternating permutation that can map to 
$(I_1, \dots, I_k)$ under $\psi_S$, and $\psi_S$ is the inverse map of 
$\phi_S$. So $\alpha_n(S)$ is equal to the number of alternating permutations 
whose swap set is contained in $S$. By the Principle of Inclusion-Exclusion, 
$\beta_n(S)$ is the number of alternating permutations whose swap set is equal to $S$.
\end{proof}

Theorem \ref{thm:ehrhartseries} follows as a corollary of Theorem \ref{thm:betanumbers}.

\begin{proof}[Proof of Theorem \ref{thm:ehrhartseries}]
Equation \ref{eqn:hstarresult} states that
\[
h^*_{\O(\Z_n)}(t) = \sum_{S \subset [n-1]} \beta_n(S) t^{\#S}.
\]
Theorem \ref{thm:betanumbers} tells us that $\beta_n(S)$ is the number of alternating permutations with swap set $S$. So the sum $\sum_{\#S = k} \beta_n(S)$ is the number of alternating permutations $\sigma$ with $\swap(\sigma) = k$. So
\[
h^*_{\O(\Z_n)}(t) = \sum_{\sigma} t^{\text{swap}(\sigma)},
\]
as needed.
\end{proof}

We conclude this section with an equidistribution result that follows as a corollary of Theorem \ref{thm:ehrhartseries}.

\begin{cor}
Let $\omega$ be a natural labeling of $\Z_n$. Then
\[
\sum_{\sigma \in A_n} t^{\swap(\sigma)} = \sum_{\sigma \in \mathcal{L}(\Z_n, \omega)} t^{\mathrm{des}(\sigma)}.
\]
\end{cor}


\section{Combinatorial Properties of Swap Numbers}\label{sec:combinatorics}

Let $s_n(k)$ denote the number of alternating permutations on $n$ letters
such that have exactly $k$ swaps.  We call these numbers the \emph{swap numbers}.  
Theorem \ref{thm:ehrhartseries}
shows that the $h^*$-polynomial of $\O(\Z_n)$ is
\[
\sum_{k = 0}^{n-1} s_n(k) t^k.
\]
We are interested in understanding these numbers.  For example, it would be interesting
to find an explicit formula for $s_n(k)$, though we have not been able to do this yet.

One straightforward property that becomes apparent looking at examples is that
$s_n(n-1) = 0$.
This is clear because it is not possible that every $k \in [n-1]$
is a swap.  Indeed, otherwise $k$ is to the left of $k+1$ for all $k \in [n-1]$
which implies that $\sigma$ is the identity permutation, which is not alternating.
Furthermore, $s_n(n-2) = 1$, since the unique alternating permutation with this many swaps
is the one with $1, 2, \ldots, \lceil \tfrac{n}{2} \rceil$ in order in the odd numbered positions
and $\lceil \tfrac{n}{2} \rceil + 1, \ldots, n$ in order in the even numbered positions.
Similarly, $s_n(0) = 1$, because there is a unique alternating permutation with
no swaps.  It is the permutation $(n-1, n, n-3, n-2, n-5, n-4, \ldots)$.  

Another property that is apparent from examples is summarized in the
following:

\begin{theorem}\label{thm:sym}
The sequence $s_n(0), s_n(1), \ldots, s_n(n-2)$ is
symmetric and unimodal.  
\end{theorem}

In fact, Theorem \ref{thm:sym} and all the preceding properties will
follow from the fact that $\O(\Z_n)$ is a Gorenstein polytope of index $3$.

\begin{definition}
An integral polytope is \emph{Gorenstein} if there is a positive integer
$m$ such that $mP$ contains exactly one lattice point $v$ in its relative interior,
and for each facet-defining inequality $a^Tx \leq b$, we have that
$b - a^Tv = 1$.  
The integer $m$ is called the \emph{index} of $P$.   
\end{definition}

See Lemma 4 (iii) in \cite{bruns2007} for this characterization of 
Gorenstein polytopes.  
The following relevant theorem concerning the $h^*$ polynomials of 
Gorenstein polytopes with unimodular triangulations is Theorem 1 
in \cite{bruns2007}.

\begin{theorem}\label{thm:bruns}
Suppose that $P$ is a  Gorenstein polytope of dimension $d$ and index $m$.
Then  $h^*_P(t)$ is a polynomial of degree $d-m+1$, whose coefficients form
a symmetric sequence.  Furthermore, the constant term of $h^*_P(t)$ is $1$. 
If, in addition, $P$ has a regular unimodular triangulation, then the 
coefficient sequence is unimodal.
\end{theorem}

\begin{proof}[Proof of Theorem \ref{thm:sym}]
It suffices to show that $\O(\Z_n)$ is a  Gorenstein polytope of index three with
a regular unimodular triangulation.
The canonical triangulation of $\O(\Z_n)$ is a regular unimodular triangulation. 
This follows from the fact the triangulation is the initial complex
of a Gr\"obner basis of the toric ideal associated to $\O(\Z_n)$.  
Indeed, this Gr\"obner basis is precisely the straightening law associated to the
Hibi ring of the distributive lattice $L( \Z_n)$ \cite{Hibi1987}. 
Initial complexes of toric ideals always yield regular triangulations \cite{sturmfels1996}.

To see that $\O(\Z_n)$ satisfies the Gorenstein property with respect to $m = 3$, note
that the defining inequalities for $3 \O(\Z_n)$ are that $v_i \geq 0$ for $i$ odd,
$v_i \leq 3$ for $i$ even, $v_{2i - 1} \leq v_{2i}$ and $v_{2i+1} \leq v_{2i}$.  
The unique interior lattice point of $3 \O(\Z_n)$ is the point $v$ where $v_i = 1$
for $i$ odd, and $v_i = 2$ for $i$ even.  Finally, this point has lattice distance
$1$ from each of the facet-defining inequalities.  Hence $\O(\Z_n)$ is a  Gorenstein polytope of index three
with a regular unimodular triangulation and Theorem \ref{thm:bruns} can be applied
to deduce that the coefficient sequence is symmetric and unimodal.  
\end{proof}

While general principles provide a proof of the symmetry and unimodality of
the sequence  $s_n(0), s_n(1), \ldots, s_n(n-2)$, it would be interesting to find
explicit combinatorial arguments that would produce these results.  In particular,
we let $A_{n,k}$ denote the set of alternating permutations on $n$ letters with 
exactly $k$ swaps, then it would be interesting to solve the following problems.

\begin{problem}
\begin{enumerate}
\item Find a bijection between $A_{n,k}$ and $A_{n,n-k-2}$.  
\item For each $0 \leq k \leq \lfloor (n -4)/2\rfloor$ find an injective map from
$A_{n,k}$ to $ A_{n,k+1}$.  
\end{enumerate}
\end{problem}

\section*{Acknowledgments}

Jane Coons was partially supported by the Max-Planck Institute 
for Mathematics in the Sciences and the US National 
Science Foundation (DGE-1746939). Seth Sullivant was partially
supported by the US National Science Foundation (DMS 1615660).

\bibliographystyle{acm}
\bibliography{shelling_bib}

\end{document}